\documentclass[11pt,a4paper]{amsart}
\usepackage{amssymb,amscd,stmaryrd}
\usepackage[top=3cm,bottom=3cm,left=2.5cm,right=2.5cm,footskip=17mm]{geometry}
\usepackage[mathcal]{eucal}
\usepackage{array,float}

\usepackage[usenames]{color}

\usepackage{amsmath, amssymb, xy}
\input xy
\xyoption{all}
\usepackage{pdflscape}
\usepackage[vcentermath]{youngtab}
\usepackage[draft]{graphicx}

\usepackage{hyperref}

\hypersetup{
    colorlinks=true,
    linkcolor=magenta,
    filecolor=cyan,      
    urlcolor=magenta,
    citecolor=magenta
}

\newcommand{\twoheaddownarrow}{\mathrel{\rotatebox[origin=c]{-90}{$\twoheadrightarrow$}}}

\numberwithin{equation}{section}

\theoremstyle{plain}

\newtheorem{thm}{Theorem}[section]
\newtheorem{proposition}[thm]{Proposition}
\newtheorem{porism}[thm]{Porism}

\newtheorem{lemma}[thm]{Lemma}

\newtheorem{thmABC}{Theorem}

\newtheorem{corABC}[thmABC]{Corollary}

\theoremstyle{definition}
\newtheorem{definition}[thm]{Definition}
\newtheorem*{remark*}{Remark}
\newtheorem{remark}[thm]{Remark}

\newcommand{\N}{\mathbb{N}}
\newcommand{\Z}{\mathbb{Z}}
\newcommand{\Q}{\mathbb{Q}}

\newcommand{\C}{\mathbb{C}}

\newcommand{\lri}{\mathfrak{o}}
\newcommand{\Lri}{\mathfrak{O}}
\newcommand{\fp}{\mathfrak{p}}
\newcommand{\Fp}{\mathfrak{P}}

\newcommand{\G}{\mathbf{G}}
\newcommand{\HH}{\mathbf{H}}
\newcommand{\GL}{\mathbf{GL}}
\newcommand{\GU}{\mathbf{GU}}
\newcommand{\SL}{\mathbf{SL}}
\newcommand{\SU}{\mathbf{SU}}

\newcommand{\g}{ \mathfrak{g}}
\newcommand{\gl}{\mathfrak{gl}}
\newcommand{\gu}{\mathfrak{gu}}

\newcommand{\jj}{ \mathfrak{j}}
\newcommand{\rr}{ \mathfrak{r}}

\newcommand{\M}{\textsf{M}}

\newcommand{\K}{\textsf{K}}

\newcommand{\tr}{\mathrm{Tr}}

\newcommand{\Gal}{\mathrm{Gal}}

\newcommand{\Stab}{\mathrm{Stab}}

\newcommand{\Ker}{\mathrm{Ker}}
\newcommand{\Hom}{\mathrm{Hom}}
\newcommand{\Irr}{\mathrm{Irr}}

\newcommand{\Ind}{\mathrm{Ind}}
\newcommand{\Res}{\mathrm{Res}}

\newcommand{\cP}{\mathcal{P}}

\newcommand{\cA}{\mathcal{A}}

\newcommand{\cM}{\mathcal{M}}
\newcommand{\cO}{\mathcal{O}}

\newcommand{\ty}{\tau}
\newcommand{\ee}{\varepsilon}

\newcommand{\kk}{k}
\newcommand{\kkq}{k_2}
\newcommand{\kkqd}{k_{2d}}
\newcommand{\kkqn}{k_{2n}}
\newcommand{\kkd}{k_d}
\newcommand{\kkn}{k_n}
\newcommand{\kkalg}{\overline{k}}

\newcommand{\lfi}{F}
\newcommand{\Lfi}{E}

\newcommand{\tC}{\mathrm{C}}

\newcommand{\xx}{{\bar{x}}}

\newcommand{\pooja}[1]{{\color{red}{#1}}}

\title[Regular characters of groups of type $\mathsf{A}_n$]{Regular characters of groups of type $\mathsf{A}_n$ over discrete valuation rings}

\author{Roi Krakovski}
\address{IBM Research - Haifa, Israel}
  \email{roik@il.ibm.com}

\author{Uri Onn}
\address{Department of Mathematics, Ben Gurion
  University of the Negev, Beer-Sheva 84105, Israel}
  \email{urionn@math.bgu.ac.il}

\author{Pooja Singla}\thanks{The research was supported by the Israel Science Foundation (ISF grant 382/11), UGC Centre
for Advanced Studies, India and by the Center for Advanced Studies in Mathematics at Ben Gurion University}
\address{Department of Mathematics,
Indian Institute of Science,
Bangalore 560012, India}
  \email{pooja@math.iisc.ernet.in}

\keywords{Representations of compact $p$-adic groups, Representation zeta functions, Ennola duality}

\subjclass[2010]{Primary 
20C15, 20G05, 11M41; Secondary  
15B33, 15B57, 20F69, 20G25, 20G35, 20H05.
}

\begin{document}


\begin{abstract}
Let $\lri$ be a complete discrete valuation ring with finite residue field $\kk$ of odd characteristic. Let $\G$ be a general or special linear group or a unitary group defined over $\lri$ and let $\g$ denote its Lie algebra. For every positive integer~$\ell$,  let $\K^\ell$ be the $\ell$-th principal congruence subgroup of $\G(\lri)$. A continuous irreducible representation of $\G(\lri)$ is called regular of level $\ell$ if it is trivial on $\K^{\ell+1}$ and its restriction to $\K^\ell/\K^{\ell+1} \simeq \g(\kk)$ consists of characters with $\G(\kk)$-stabiliser of minimal dimension. In this paper we construct the regular characters of $\G(\lri)$, compute their degrees and show that the latter satisfy Ennola duality. We give explicit uniform formulae for the regular part of the representation zeta functions of these groups.

\end{abstract}
\maketitle


\section{Introduction}

\subsection{Overview}
Let $\lfi$ be a non-Archimedean local field with ring of integers $\lri$, maximal ideal $\fp$ and finite residue field $\kk$ of cardinality $q$ and characteristic $p$. Thus $\lfi$ is either the field $k(\!(t)\!)$ of formal Laurent series in one variable over~$k$ or a finite extension of the field $\Q_p$ of $p$-adic numbers. For a positive integer $n$ let $\G$ be either the general linear group $\GL_n$ or the standard unitary group $\GU_n$ attached to an unramified quadratic extension $\Lfi/\lfi$. 

In this paper we construct a family of continuous complex-valued irreducible representations of the groups~$\G(\lri)$ and $\G'(\lri)=\G(\lri) \cap \SL_n(E)$ and analyse their dimensions. Being a maximal compact subgroup of $\G(\lfi)$, the group~$\G(\lri)$ and its representations play an important role in the construction of admissible representations of $\G(\lfi)$; see, for instance,~\cite{BK}. In a different direction, representations of the groups $\G'(\lri)$ play an important role in the representation theory of arithmetic groups of type $\textsf{A}_{n-1}$.  Indeed, suppose that $\cO$ is the ring of integers in a global field~$L$ and that $\HH$ is a connected, simply-connected absolutely almost simple algebraic group of type~$\mathsf{A}_{n-1}$ defined over~$\cO$. Assume that $\HH(\cO)$ has the congruence subgroup property, which means that finite quotients thereof are exhausted by the quotients of $\HH(\cO/I)$, where $I$ ranges over the non-zero ideals of $\cO$. Then, the finite dimensional irreducible complex representations of $\HH(\cO)$ are finite tensor products of representations of the completions $\HH(\cO_v)$, $v$ a place of $L$; see~\cite{LL}. For all but finitely many places these completions are of the form $\G'(\lri)$.     

In contrast to the well-understood representation theory of the finite groups of Lie type $\G(k)$ or of the locally compact groups $\G(F)$, representations of $\G(\lri)$ are considerably less understood; see \cite{Stasinski-survey} for a comprehensive account of the state of the art for the case of general linear groups.

The focus of this paper is on regular representations of $\G(\lri)$, which we now discuss. Recall that by virtue of the profiniteness of $\G(\lri)$, every complex continuous irreducible representation $\rho$ of $\G(\lri)$ has a {\em level} $\ell=\ell(\rho)$; this is the smallest natural number $\ell$ such that $\rho$ factors through the principal congruence quotient~$\G(\lri/\fp^{\ell+1})$. Let $\g$ denote the Lie algebra scheme of $\G$ and, for every positive integer $m$, let $\K^m=\mathrm{Ker}\left(\G(\lri) \to \G(\lri/\fp^m)\right)$be the~$m$-th principal congruence subgroup of $\G(\lri)$.  One has $\G(\kk)$-equivariant isomorphisms $\K^m/\K^{m+1} \simeq (\g(\kk),+)$ and a $\G(\kk)$-equivariant isomorphism $x \mapsto \varphi_x$ between $\g(\kk)$ and its Pontryagin dual 
 $\g(\kk)^\vee$= $\mathrm{Hom}_\Z\left(\g(\kk),\mathbb{C}^\times\right)$; see Section \ref{subsec:duality.for.lie.rings}.  
 \begin{definition}(Regular representations and characters) A character $\varphi_x \in \g(\kk)^\vee$ is called {\em regular} if $x$ is a regular element of $\g(\kk)$, that is, the characteristic polynomial of $x$ is equal to its minimal polynomial. An irreducible representation of level $\ell$ of $\G(\lri)$ and its character are called {\em regular} if the orbit of their restriction to $\K^\ell$ consists of pullbacks of regular characters of $\K^\ell/\K^{\ell+1} \simeq \g(\kk)$. 
\end{definition}

Regular elements of $\G(\kkalg)$, where $\kkalg$ is an algebraic closure of $\kk$, are defined as elements whose centraliser has minimal dimension, see e.g.  \cite[\S 3.5]{MR0352279}. We extend this definition to the present situation and consider the adjoint action of $\G(\kk)$ on $\g(\kk)$. In this case the stabilisers of regular elements have minimal dimension, namely, the rank of the group $n = \mathrm{rank}(\G)$; they are abelian; and the corresponding matrix is cyclic. See \cite[Theorem~3.6]{HGreg} for a proof of these equivalences over $\kkalg$ and deduce them for $\kk$ using base change.

\subsection{Main results} 
Our first main result is a construction of the regular characters of $\G(\lri)$ using explicit characters of their principal congruence subquotients.

\begin{thmABC}\label{thm:construction} Let $\G$ be $\GL_n$ or $\GU_n$ and let $\lri$ be a complete discrete valuation ring with finite residue field of odd characteristic. Let $\ell>1$ be an integer and let $m=\lfloor \ell/2 \rfloor$. For every regular character $\chi$ of $\G(\lri)$ of level $\ell-1$ there exists an explicitly constructed character $\sigma$ of $\K^m/\K^{\ell}$ such that $\sigma$ extends to its stabiliser in $\G(\lri)$ and the extended character induces irreducibly to give~$\chi$.
\end{thmABC}

The study of regular characters for the groups $\GL_n(\lri)$ was initiated by Hill in~\cite{HGreg, HGreg-SSandCusp}. Hill constructed regular characters associated to split regular matrices in $\gl_n(\kk)$. As noted by Takase in~\cite{MR3448171}, Hill's argument fails in non-semisimple cases and a fix that depends on a conjectural vanishing of a certain Schur multiplier is given in \cite{MR3448171}. Although the present paper follows a somewhat different path, we encounter the same multiplier and show that the conjecture \cite[Conjecture~4.6.5]{MR3448171} indeed holds and in much greater generality; see Theorem~\ref{thm:bijection} and its proof. Regular characters of $\GL_n(\lri)$ that correspond to irreducible matrices in $\gl_n(\kk)$ were also constructed in~\cite{AOPS}. Recently, Stasinski and Stevens constructed the regular characters of $\GL_n(\lri)$, see~\cite{StaSte}. Their construction follows a different approach than the one taken here and applies to even residual characteristic as well.

\smallskip

Our next result gives uniform formulae for the degrees of the regular characters. We call a matrix $\ty \in \M_n(\Z_{\ge 0})$  an $n$-{\rm type} if  $n=\sum_{d,e} d e 
\ty_{d,e}$. Let $\cA_n$ be the set of $n$-types. It turns out (see Section~\ref{sec:regulars.in.g.k}) that regular elements in $\g(\kk)$ can be partitioned according to $n$-types and the latter determine the former's centralisers and eventually also determine the degrees of the corresponding regular characters. For $\epsilon \in \{ \pm 1\}$ and $\ty \in \cA_n$ define the following polynomials in~$\Z[x]$ 
\[
\begin{split}
 v_\epsilon(x)&=x^{n^2}  \prod_{1 \le d \le n}\left(1-\epsilon^{d}x^{-d}\right),~\text{and}  \\
  u^\ty_\epsilon(x)&=x^{n}  \prod_{1 \le d,e \le n} \left(1-\epsilon^d x^{-d}\right)^{\ty_{d,e}}.
\end{split}
\]
Let 
\[
\ee_\G=
\left\{
\begin{array}{cc}
1 & \text{if $\G=\GL_n$},\\
-1 & \text{if $\G=\GU_n$}.
\end{array} \right.
\]
We remark that $v_{\ee_\G}(q)$ is the order of $\G(\kk)$ and $u_{\ee_\G}^\ty(q)$ is the order of the centraliser in $\G(\kk)$ of a regular element in $\g(\kk)$ of type $\tau$; see Proposition~\ref{prop:form.of.centralizers}. 

\begin{thmABC}\label{thm:dimensions} 
Let $\G$ be either $\GL_n$ or $\GU_n$ and let $\ee_\G$ be $+1$ or $-1$, respectively. 
Then, for every complete discrete valuation ring $\lri$ with finite residue field of cardinality $q > n$ and odd characteristic, the degrees of 
the regular characters of $\G(\lri)$ of level $\ell$ are
\[
q^{\left(n \atop 2\right)(\ell-1)}v_{\ee_\G}(q)/u_{\ee_\G}^\ty(q), \quad \ty \in \cA_n.
\]
In particular, the degrees of the regular characters satisfy Ennola duality: for every polynomial $f \in \Z[x]$ the evaluation $f(q)$ is the degree of a regular character of $\GL_n(\lri)$ if and only if  $(-1)^{\deg(f)}f(-q)$ is the degree of a regular character of $\GU_n(\lri)$ of the same level. 
\end{thmABC}
We remark that the character degrees provided by Theorem~\ref{thm:dimensions} form a superset of the possible degrees for $q \le n$ as well, but in such case not all of them occur. The reminiscence in Theorem~\ref{thm:dimensions} of 'Ennola duality' gives further evidence that this phenomenon features in the groups $\GL_n(\lri)$ and $\GU_n(\lri)$ in general. In \cite{Ennola-unitary-char} Ennola observed that characters of $\GL_n(\lri/\fp)$ and of $\GU_n(\lri/\fp)$ are related upon a formal substitution of '$-q$', see also \cite{TV}. In fact, all characters of groups of type $\mathsf{A}_1$ and $\mathsf{A}_2$ obey this duality, as observed in \cite{AKOV2}, and also characters of level one of groups of type~$\mathsf{A}_3$, as observed in~\cite{Levy}.

\smallskip

The next result gives the number of regular characters of type $\ty$ and level $\ell$. For $d \in \N$, let
\[
w_d(x)=\frac{1}{d}\sum_{m|d}\mu(d/m)x^{m} \in \Q[x],
\]
where $\mu$ is the M\"{o}bius function. Thus $w_d(q)$ is the number of monic irreducible polynomials of degree $d$ over $k$. Let $\left(a \atop b \right)$ and $\left(a \atop b_1, b_2, \ldots, b_r \right)$ denote the binomial and multinomial coefficients, respectively. In what follows all summations and products indexed by $d$ and $e$ range over the set $\{1,...,n\}$. 

\begin{thmABC}\label{thm:numbers} Let $\G$ be either $\GL_n$ or $\GU_n$ and let $\ee_\G$ be $+1$ or $-1$, respectively. 
Then, for every complete discrete valuation ring $\lri$ with finite residue field of cardinality $q > n$ and odd characteristic, the number of regular characters of $\G(\lri)$ of type $\ty \in \mathcal{A}_n$ and level $\ell \in \N$  is   
\[
q^{(\ell-1) n}u_{\ee_\G}^\ty(q) \prod_{d} \left(\sum_{e} \tau_{d,e} \atop \tau_{d,1}, \tau_{d,2}, \ldots\right) \left(w_d(q) \atop \sum_e \tau_{d,e}\right).
\]
\end{thmABC}

A group $\Gamma$ is called {\em rigid} if for every $m \in \N$ there is a finite number, denoted $r_m(\Gamma)$, of inequivalent complex $m$-dimensional irreducible representations of $\Gamma$. The group $\Gamma$ has {\em polynomial representation growth} if the sequence $r_m(\Gamma)$ is bounded by a polynomial in~$m$. One then defines the representation zeta function of $\Gamma$ to be the Dirichlet generating function $\zeta_{\Gamma}(s)=\sum_m r_m(\Gamma) m^{-s}$. If $\Gamma$ is topological one counts only continuous representations. The group $\G(\lri)$ is not rigid; for instance, there are infinitely many linear continuous characters factoring through the determinant map. However, the number of regular representations of $\G(\lri)$ of dimension $m$ is bounded by polynomial in~$m$. Following Stasinski\footnote{personal communication.} we look at the regular part of the representation zeta function of the groups~$\G(\lri)$, the Dirichlet generating function that counts only regular characters. Combining Theorems~\ref{thm:dimensions}~and~\ref{thm:numbers} gives an explicit computation of the regular zeta function. 

\begin{corABC}\label{cor:zeta.G} The regular part of the representation zeta function of $\G(\lri)$ is 
\[
\zeta^{\mathrm{reg}}_{\G(\lri)}(s)=\frac{1}{1-q^{n-\left(n \atop 2\right) s}} \sum_{\ty \in \cA_n} u_{\ee_\G}^\ty(q) \prod_{d } \left(\sum_{e} \tau_{d,e} \atop \tau_{d,1}, \tau_{d,2}, \ldots\right) \left(w_d(q) \atop \sum_e \tau_{d,e}\right) \left( \frac{ v_{\ee_\G}(q)}{u_{\ee_\G}^\ty(q) }  \right)^{-s}.
\]

\end{corABC}

Finally, we give an explicit computation of the regular representation zeta function associated with $\G'(\lri)$. For $\ty \in \cA_n$ and $r \in \N$ let
\begin{equation}
\iota(\ty, r)=\gcd(\{e \mid \exists d ~\mathrm{with}~ \ty_{d,e} \ne 0 \} \cup \{r\}).
\end{equation}

\begin{thmABC}\label{thm:zeta.H} Let $\G'$ be either $\SL_n$ or $\SU_n$ and let $\ee_{\G}$ be $+1$ or $-1$, respectively. 
Then, for every discrete valuation ring $\lri$ with residue field of odd characteristic prime to $n$ and cardinality $q > n$, the regular part of the representation zeta function of $\G'(\lri)$ is 
\[
\begin{split}
\zeta^{\mathrm{reg}}_{\G'(\lri_{\ell+1})}(s) &=\frac{1}{1-q^{n-1-\left(n \atop 2\right) s}}\sum_{\ty \in \cA_n} \frac{\iota(\ty,q-\ee_\G)^2}{(q-\ee_\G)} u_{\ee_\G}^\ty(q) \\ & \qquad \qquad \qquad \times q^{-1}\prod_{d } \left(\sum_{e} \tau_{d,e} \atop \tau_{d,1}, \tau_{d,2}, \ldots\right) \left(w_d(q) \atop \sum_e \tau_{d,e}\right) \left(   \frac{v_{\ee_\G}(q)}{u_{\ee_\G}^\ty(q)\iota(\ty,q-\ee_\G)}   \right)^{-s}.
\end{split}
\]

\end{thmABC}

A remarkable feature of the enumerative and quantitative results above, compactly expressed by the various zeta functions,  is their uniformity.  In case $\G$ is one of the forms $\GL_n$ or $\GU_n$  the zeta function depends only on the form (through $\ee_\G$) and on the cardinality $q$ of the residue field, where the latter dependence is merely through substitution of $q$ in a finite explicit list of polynomials. In the case of $\SL_n$ or $\SU_n$ one still gets uniform behaviour with some additional sensitivity to divisors of $q-\ee_\G$.  

\subsection{Organisation of the paper} Section~\ref{sec:notation} contains notation and preliminary remarks. In Section~\ref{sec:construction} we prove an elaboration of Theorem~\ref{thm:construction}. In Section~\ref{sec:regulars.in.g.k} we describe regular elements in $\g(\kk)$ and their centralisers. In Section~\ref{sec:sim.enum.ennola}  we prove the remaining results mentioned in the introduction.

\subsection{Acknowledgements} We thank Alexander Stasinski and Shaun Stevens for sharing with us a preliminary version of their paper on regular representations. We warmly thank them and Christopher Voll for carefully reading preliminary versions of this paper and providing valuable and constructive feedback.

 
\section{Notation and preliminary remarks}\label{sec:notation} 
For a group $G$ and $g,h \in G$ we denote the group commutator $ghg^{-1}h^{-1}$ by $(g,h)$. For a Lie ring $L$ we use $[a,b]$ to denote the Lie bracket of $a,b \in L$. All our Lie rings are contained in matrix rings, in particular, the Lie bracket is induced from the associative structure, that is $[a,b]=ab-ba$. 

\subsection{Rings and fields} The rings and fields that we consider fit into the following diagram 
\[
\begin{matrix}
  \Lfi &  \supset & \Lri& \triangleright & \Fp=\pi\Lri \\
       \vert & & \vert & & \vert \\
   \lfi &  \supset & \lri & \triangleright & \fp=\pi \lri
\end{matrix}
\qquad\qquad
\begin{matrix}
  \kkq  \\ \vert  \\ \kk 
\end{matrix}
\]
Here $\lfi$ is non-Archimedean local field with ring of integers $\lri$, maximal ideal $\fp$, residue field~$\kk$ and  a fixed uniformiser $\pi$. The field $\Lfi$ is an unramified quadratic extension of $F$ with ring of integers $\Lri$. We let $\nu \in \Gal(\Lfi/\lfi)$ be the nontrivial Galois automorphism.  The residue fields have characteristic $p$ which is assumed to be odd (except in Section~\ref{sec:regulars.in.g.k}). For $\ell \in \N$ we let $\lri_\ell=\lri/\fp^\ell$ and similarly $\Lri_\ell=\Lri/\Fp^\ell$ be the finite quotients. For $m<\ell$ we have natural reduction maps $\M_n(\Lri_\ell) \to \M_n(\Lri_m)$ which we denote by $x \mapsto x_m$.

\subsection{Groups and Lie algebras} Let $\G$ denote either the general linear group or the following unitary group  associated to the extension $E/F$. Let $\circ : \M_n(\Lfi) \to \M_n(\Lfi)$ be the involution on $n \times n$ matrices defined by $(a_{ij})^\circ \mapsto (a_{ji}^{\nu})$. Then $\GU_n$ is the affine
algebraic group defined by 
the equations $x^\circ x =I$.  Let $\g$ be the Lie algebra scheme of $\G$. Then $\g$ is either $\gl_n$ or the anti-hermitian Lie
algebra $\gu_n$ defined by the equations $x+x^\circ=0$.
Let $\G'$ and $\g'$ stand for $\G \cap \SL_n$ and $\g \cap \mathfrak{sl}_n$, respectively.

\subsection{Duality for Lie rings}\label{subsec:duality.for.lie.rings}  The equivariant duality between $\g(\lri_\ell)$ and its Pontraigin dual
$\g(\lri_\ell)^\vee:=\Hom_\Z(\g(\lri_\ell),\C^\times)$ is established as follows. We fix a non-trivial additive character 
\begin{equation}\label{eq:def.of.psi}
\psi:\Lfi \to \C^\times
\end{equation} 
with $\mathrm{Ker}(\psi) = \Lri$, see e.g. \cite[Lemma~5.12 and Remark~5.13]{AKOV2}. 
We get a well-defined map
\[
\M_n(\Lri_\ell) \rightarrow \M_n(\Lri_\ell)^\vee, \quad x \mapsto \varphi_x, ~\text{where} ~\varphi_x(y)=\psi
(\pi^{-\ell}\tr(xy)).
\]
Moreover, as $\Fp^{-1} \not \subset \Ker (\psi)$, this map is an isomorphism and in fact restricts to an isomorphism $\g(\lri_\ell) \simeq \g(\lri_\ell)^\vee$ which is $\G(\lri_\ell)$-equivariant with respect to the adjoint and co-adjoint actions, see \cite[Lemma~5.12]{AKOV2}. We call a character $\varphi_x$ {\em regular} if the corresponding matrix $x$ is regular (cyclic). Note that the notion of regularity is independent of the level, namely, $x \in \M_n(\Lri_\ell)$ is a cyclic matrix if and only if any of its reductions modulo $\pi^m$ is cyclic. In such case $\tC_{\M_n(\Lri_\ell)}(x)=\Lri_\ell[x]$, namely, the centraliser of $x$ is the $\Lri_\ell$-algebra generated by $x$.

\subsection{Representations and characters} All the groups we consider are compact or moreover finite. We let $\Irr(G)$ denote the set of  characters of  continuous irreducible representations of $G$. As our groups are totally disconnected we have, in particular, that $\Irr(G)=\bigcup \Irr(G/N)$, where the union is over the finite index normal subgroups of $G$. If~$N$ is a normal finite index subgroup of $G$ and $\sigma \in \Irr(N)$ we write $I_G(\sigma)$ for the inertia group of $\sigma$ in $G$. We write $\Irr(G \mid \sigma)$ for the subset of irreducible characters of $G$ whose restriction to $N$ contains~$\sigma$.



\section{Construction of the regular representations of $\G(\lri)$}\label{sec:construction}

Throughout this section it will be convenient to descend to finite quotients and discuss representations of $\G(\lri_\ell)$ rather than representations of $\G(\lri)$ of level $\ell-1$.  For $1 \le m < \ell$ let  $\K^m_\ell=\K^m/\K^\ell=\Ker \left(\G(\lri_\ell) \to \G(\lri_m)\right)$.
The main result in this section is the following theorem which is an elaborate version of Theorem~\ref{thm:construction}.

\begin{thm}\label{thm:bijection} Let $\G$ be either $\GL_n$ or $\GU_n$ and let $\g=\mathrm{Lie}(\G)$. 
Let $\lri$ be a complete discrete valuation ring with finite residue field of odd characteristic. Let $\ell > 1$ be an integer and set $m=\lfloor \ell/2 \rfloor$. Let $\sigma$ be an irreducible character of $\K^m_\ell$ lying above a regular character of~$\K^{\ell-1}_\ell$. Then the following hold.

\begin{enumerate}

\item [(1)] The character 
$\sigma$ extends to its inertia group $I_{\G(\lri_\ell)}(\sigma)$. In particular, every such extension induces to a regular character of $\G(\lri_\ell)$. 

\item[(2)] 
The character $\sigma$ corresponds to an orbit of a regular matrix $x \in \g(\lri_{\ell-m})$ and there exists a bijection   
  \[
\Irr\left( \G(\lri_\ell) \mid \sigma \right) \simeq \tC_{\G(\lri_{m})}(x_{m})^\vee,
    \]
  where $x_m$ is the image of $x$ in $\g(\lri_m)$.

\end{enumerate}
In particular, every regular representation of $\G(\lri)$ arises in this manner for appropriate $\ell$ and~$\sigma$.

\end{thm}

The proof of Theorem~\ref{thm:bijection} generalises and modelled after  \cite[Section 7]{Jaikin}. In {\em loc}.\ {\em cit}.\  Jaikin-Zapirain analyses the irreducible characters of $\SL_2(\lri)$ of positive level which in this case are all regular. We require some preparations recorded in Sections~\ref{subsec:exp.and.log},~\ref{subsec:reps.of.K}, \ref{subsec:form.of.stab}~and~\ref{subsec:stable.max.isot} that allow us to prove Theorem~\ref{thm:bijection} in Section~\ref{subsec:proof.of.thm.bijection}. 

\subsection{Exponent and logarithm}\label{subsec:exp.and.log} 

 If $\lri$ has characteristic zero, the exponent and logarithm defined by the power series $\exp(x)=\sum_{r=0}^{\infty} x^r/{r !}$ and $\log(1+x)=\sum_{r=\pooja{1}}^{\infty} (-1)^rx^r/r$, respectively, define mutually inverse homeomorphisms between $\K^m$ and $\g(\pi^m \lri)$. Such a bijection fails in positive characteristic. Nevertheless, by controlling the ratio $\ell/m$ one obtains a bijection of finite subquotients in positive characteristic as well. Indeed, $\exp$ and $\log$ truncate to polynomials of degree $\lfloor (\ell-1)/m \rfloor$ and establish a bijection $\K^m_\ell \simeq \g(\pi^m\lri_\ell)$ whenever $p=\mathrm{char}(\kk) > \lfloor (\ell-1)/m \rfloor$.   

Two important special cases arise. For $\ell/2 \le m <  \ell$ the exponential map reduces to a linear polynomial ${\mathrm{exp}(x)=1+x}$ giving rise to an isomorphism of abelian groups ${\mathrm{exp}:\g(\pi^m\lri_{\ell}) \to \K_\ell^m}$. If $\ell/3 \le m <\ell/2$, the exponential map reduces to a quadratic polynomial ${\exp(x)=1+x+\frac{1}{2}x^2}$, establishing a bijection $\g(\pi^m\lri_{\ell}) \to \K_\ell^m$.  Here enters the main reason for our assumption throughout that $p >2$, though not the only one. In such case the bijection is no longer an isomorphism of groups but is not 'too far' from being an isomorphism in the sense made precise in the following Lemma (see \cite[Section 7]{Jaikin}).    

\begin{lemma}\label{lem:isom.of.modules}
Let $\ell$ and $m$ be positive integers such that $\ell/3 \le m < \ell$ and let $(M,\star)$ denote the $\G(\lri_\ell)$-module defined as follows. As a set $M=\K^m_\ell$. The abelian group structure is given by $a \star b := ab(a,b)^{-1/2}$ for all $a,b \in M$ and the $\G(\lri_\ell)$-action is the conjugation action. Then $\exp: (\g(\pi^m\lri_\ell),+) \to (M,\star)$ is an isomorphism of $\G(\lri_\ell)$-modules.

\end{lemma} 

\begin{proof}

Whenever $\exp$ and $\log$ converge they are mutual inverses. Therefore, our assumption on $m$, $\ell$ and the characteristic of $\Lri/\Fp$  guarantees that they establish a bijection between $\M_n(\pi^m\Lri_\ell)$  and $1+\M_n(\pi^m\Lri_\ell)$. Furthermore, since $\exp$ and $\log$ commute with the involution $\circ$ they restrict to a bijection $\g(\pi^m\lri_\ell) \simeq \K^m_\ell$.  Our assumption on the characteristic also implies that squaring is an automorphism of the subgroup $(\K^m_\ell,\K^m_\ell) \subset \K^{2m}_\ell \simeq (\g({\lri_{\ell-2m}}),+)$ so that the formula for $\star$ is well-defined. The commutativity of $\star$ follows from $(1+\pi^m y)\star(1+\pi^m z)= 1+\pi^m(y+z)+\frac{1}{2} \pi^{2m}(yz+zy)$ modulo~$\pi^\ell$.  Finally, conjugation commutes with $\exp$ and $\log$, therefore, this is an isomorphism of $\G(\lri_\ell)$-modules where the action factors through $\G(\lri_{\ell-m})$.  \qedhere

\end{proof}

\subsection{Characters of $\K^{\lfloor \ell /2 \rfloor}_\ell$}\label{subsec:reps.of.K} If $\ell=2m$ then $\K^m_{\ell} \simeq \g(\pi^m\lri_{\ell})$, consequently, the irreducible characters of $\K^m_{\ell}$ are linear characters of the form $\log^*\!\theta$ for $\theta \in \g(\pi^m\lri_{\ell})^\vee$. 
If $\ell=2m+1$ the group $\K^m_{\ell}$ is a $p$-group of nilpotency class two and its irreducible characters are well understood from various perspectives, see e.g. \cite[Proposition 8.3.3]{BF}. We describe it in terms of the present setup, as depicted in Figure~\ref{fig:1}. 
\begin{figure}[htb!]
\caption{}
\centering
 \label{fig:1}
\begin{displaymath}
\xymatrix{
\g(\kk) \ar@{-}[d] & \ar[l] \g(\pi^m\lri_\ell)   \ar@{-}[d] &  \ar[l]_{\quad \log} {\K^m_\ell}\ar@{-}[d]   \\
\overline{\jj} \ar@{-}[d] & \ar[l] {\mathfrak{j}} \ar@{-}[d] &\ar[l]_{\quad \log}  {J}\ar@{-}[d] \\
 \overline{\rr}_\xx \ar@{-}[d]  & \ar[l] {\mathfrak{r}}_\xx \ar@{-}[d]& \ar[l]_{\quad \log}   {R_\xx}\ar@{-}[d]  &          \\
(0) &\ar[l] \g(\pi^{m+1}\lri_\ell) & \ar[l]_{\quad \log}    {\K^{m+1}_\ell}
}
\end{displaymath}
 \end{figure}

The subgroup $\K^{m+1}_\ell \subset \K^m_\ell$ is central and each of its characters are of the form $\log^*\!\theta$ for some $\theta \in \g(\pi^{m+1}\lri_\ell)^\vee$. Anticipating the sequel it is convenient to think of $\theta$ as the restriction of $\varphi_x$ for some $x \in \g(\lri_\ell)$. Let
\[
B_{\theta}: \K^{m}_\ell/\K^{m+1}_\ell \times \K^{m}_\ell/\K^{m+1}_\ell \rightarrow \mathbb C^{\times}; \hspace{0.5cm} B_{\theta}(a\K^{m+1}_\ell, b\K^{m+1}_\ell) = \theta\left(
\mathrm{log}(a,b)\right).
\]
Then $B_{\theta}$ is an alternating bilinear form on $\K^{m}_\ell/\K^{m+1}_\ell \simeq \g(\kk)$ with values in $\C^\times$ rather than $\kk$. Unravelling definitions it can be described in terms of the alternating bilinear form ${\beta_{\xx} (\bar{y},\bar{z})=\mathrm{Tr}(\bar{x}[\bar{y},\bar{z}])}$ that makes the diagram 
\[
\begin{matrix}
\K^{m}_\ell/\K^{m+1}_\ell & \times & \K^{m}_\ell/\K^{m+1}_\ell & \overset{B_\theta}{\longrightarrow} & \C^\times & \\
\mid \! \wr  & & \mid \! \wr & & \uparrow & \!\! \psi\left(\pi^{-1}(\cdot)\right)  \\
\g(\kk) & \times & \g(\kk)   & \overset{\beta_\xx}{\longrightarrow} & \kk &
\end{matrix}
\]
commute, where $\xx \in \g(\kk)$, the reduction of $x$ modulo $\pi$, represents the restriction of $\theta$ to $\g(\pi^{\ell-1}\lri_\ell)$. In particular $B_\theta$ depends only on $\xx$. 
Let $\overline{\mathfrak{r}}_\xx$ denote the radical of $\beta_\xx$. Then 
\[
\begin{split}
\overline{\mathfrak{r}}_\xx&=\{\bar{y} \in\g(\kk) \mid \mathrm{Tr}(\bar{x}[\bar{y},\bar{z}])=0,~ \forall \bar{z} \in \g(\kk)\}   \qquad \text{(by definition)}\\
&=\tC_{\g(\kk)}(\bar{x}) \qquad \qquad \qquad \quad \qquad \qquad \qquad  \qquad \text{(by the non-degeneracy of the trace)}\\
&={\mathfrak{r}}_\xx/\g(\pi^{m+1}\lri_\ell) \simeq R_\xx/\K^{m+1}_\ell,
\end{split}
\]
where ${\mathfrak{r}}_\xx \subset \g(\pi^m\lri_\ell)$ is the inverse image of $\overline{\mathfrak{r}}_\xx$ under  reduction modulo $\pi^{m+1}$ and $R_{\xx}=\exp \left({\mathfrak{r}}_\xx  \right)$. The character $\theta$ can be extended to $\mathfrak{r}_\xx$ in $|\overline{\mathfrak{r}}_\xx|$-many ways. Each such extension $\theta'$ determines uniquely an irreducible character of $\K^m_\ell$ that arises as follows. The induced bilinear form on $\g(\kk)/\overline{\mathfrak{r}}_\xx$  is non-degenerate. Let $\overline{\mathfrak{j}} \supset \overline{\mathfrak{r}}_\xx$ be a maximal isotropic subspace and let $\mathfrak{j}$ be its inverse image in $\g(\pi^m\lri_\ell)$. Let $\theta''$ be an extension of $\theta'$ to $\mathfrak{j}$. Then
\begin{enumerate}

\item[(i)] The pullback $\log^*\!\theta''$ is a character of $J=\exp(\jj)$.

\item[(ii)] The character $\sigma=\Ind_{J}^{\K^m_\ell}(\log^*\!\theta'')$ is irreducible and independent of  $\jj$ and $\theta''$. 

\item[(iii)] There is a bijection between extensions $\theta'$ of $\theta$ to $\rr_\xx$ and $\Irr(\K^m_\ell \mid \theta)$.

\end{enumerate}
These are well known assertions coined as 'Heisenberg lift'. We remark that (i) follows from 
\[
\log(ab)=\log(a\star b)+\log (a,b)^{1/2}=\log(a)+\log(b)+\log (a,b)^{1/2}, \quad a,b \in \K^m_\ell,
\]
where the first equality follows from the definition of $\star$ and the fact that $(a,b)^{1/2}$ is central; the second equality follows from the fact that $\log$ is an isomorphism of modules, see Lemma~\ref{lem:isom.of.modules}. Applying $\theta''$ and restricting $a,b$ further down to be in $J$ we get  
\[
\left[\theta''\circ \log\right](ab)= \left[\theta'' \circ \log\right](a)\cdot \left[\theta''\circ \log\right](b), 
\]
since $\log (a,b)^{1/2}$ is in the kernel of $\theta$.

\subsection{Stabilisers in $\G(\lri_\ell)$ of regular characters}\label{subsec:form.of.stab}
We now recall the structure of stabilizers of regular characters which have a particularly simple description.
This is recorded in the following Lemma, the proof of which in the $\GL_n$ case is found 
in~\cite[Corollary 3.9]{HGreg}. The same proof works verbatim for the unitary case and is omitted.          

\begin{lemma}\label{lem:form.of.stab} Let $x \in \g(\lri_\ell)$ be a regular element and let $\varphi=\varphi_x \in \g(\lri_\ell)^\vee$ be the corresponding regular character. Let $r < \ell/2$ be a positive integer, let $\varphi_{ \mid \ell-r}$ denote the restriction of $\varphi$ to~$\g(\pi^{\ell-r}\lri_\ell)$ and let $\log^*(\varphi_{\mid {\ell-r}})$ be its pullback to $K^{\ell-r}_\ell$ along the $\log$ map. Then 
\[
I_{\G(\lri_\ell)}\left(\log^*(\varphi_{\mid {\ell-r}})\right)=\Stab_{\G(\lri_\ell)}(x_{\ell-r})=\tC_{\G(\lri_\ell)}(x) \K^r_\ell. 
\]
\end{lemma}

Lemma~\ref{lem:form.of.stab} reveals the main feature of regular characters that simplifies their analysis; their stabilisers can be described in terms of a fixed group, namely the centraliser of a matrix  $\tC_{\G(\lri_\ell)}(x)=\Lri_\ell[x] \cap \G(\lri_\ell)$,  together with a principal congruence subgroup.

\subsection{Stable maximal isotropic subspace}\label{subsec:stable.max.isot} We require the following standard lemma (stated as Lemma~2.4 in \cite{Jaikin} and referenced to the proof of \cite[Lemma 1.4]{Howe}). For completeness we give a simple proof.

\begin{lemma}\label{lem:stable.max.iso} Let $V$ be a finite dimensional $\mathbb{F}_p$-vector space and $\alpha:V \times V \to \mathbb{F}_p$ an antisymmetric bilinear form on V. Suppose that a $p$-group $P$ is acting on $V$ and preserving~$\alpha$. Then there exists a maximal isotropic subspace $U$ of $V$ which is $P$-invariant.
\end{lemma}

\begin{proof} We use induction on $\dim V$.  If $\dim V =1$ then $\alpha$ is the zero form and there is nothing to prove. Since $|V|$ and $|P|$ are powers of $p$ and $0 \in V$ is a fixed point for $P$,  there is a nonzero $P$-fixed vector $v \in V$. Let $W$ be the subspace of all vectors in $V$ that are orthogonal to $v$. Since $v \in W$ we may look at the quotient $\overline{W}:=W/\mathbb{F}_pv$.  The form $\alpha$ descends to a form $\overline{\alpha}$ on $\overline{W}$ and is preserved by the induced $P$-action. By induction there is a maximal isotropic $P$-stable subspace $\overline{U}$ in $\overline{W}$. Its inverse image in $V$ is the required subspace for $(V,\alpha)$.   \qedhere

\end{proof}

\subsection{Proof of Theorem~\ref{thm:bijection}}\label{subsec:proof.of.thm.bijection} 
The proof for the case $\ell=2m$ is straightforward. Let $\sigma$ be an irreducible character of $\K^m_\ell$ lying above a regular character of $\K^{\ell-1}_\ell$, that is, $\sigma$ is  the pullback under $\log$ of a restriction of $\varphi=\varphi_x$ for a regular element $x$ in $\g(\lri_\ell)$.

\begin{figure}[htb!]
\centering
 \caption{Even case $\ell=2m$}
\label{fig:ell.even}
\begin{displaymath}
\xymatrix{
    &    &     {\G(\lri_\ell)}\ar@{-}[d]   &      \\
& & {I_{\G(\lri_\ell)}(\sigma)}\ar@{-}[dr]\ar@{-}[dl] & \\ 
 \g(\lri_{m}) \simeq {\g(\pi^{m}\lri_\ell)} & \ar[l]_{\quad \qquad \log} {\K^{m}_\ell}\ar@{-}[dr]& & \tC_{\G(\lri_\ell)}(x)\ar@{-}[dl]\\
 & & \K^m_\ell \cap \tC_{\G(\lri_\ell)}(x) \ar@{-}[d] &\\
&&{\{1\}}& }
\end{displaymath}
\end{figure}
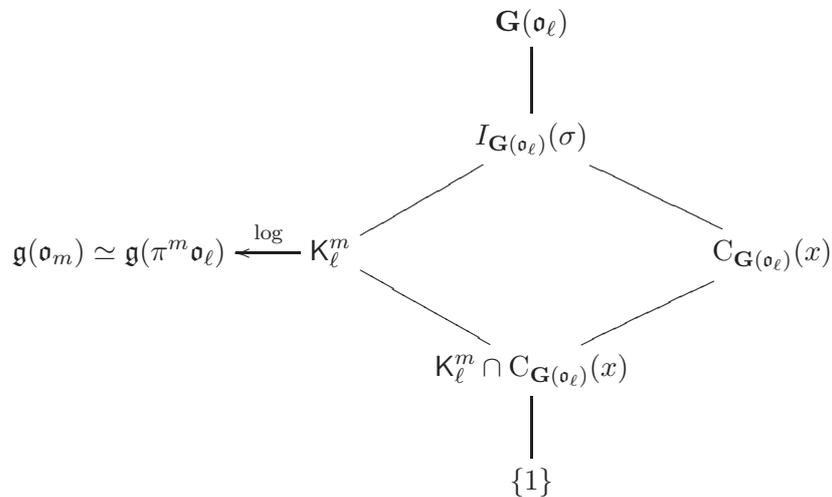

Lemma~\ref{lem:form.of.stab} guarantees that $I_{\G(\lri_\ell)}(\sigma)=\tC_{\G(\lri_\ell)}(x)\K^m_\ell$ is a product of abelian groups, where $\tC_{\G(\lri_\ell)}(x)$ normalises $\K^m_\ell$ and stabilises $\sigma$. The various subgroups are depicted in Figure~\ref{fig:ell.even}. 
It follows that every character of $\tC_{\G(\lri_\ell)}(x)$ that agrees with $\sigma$ on 
$\K^m_\ell \cap \tC_{\G(\lri_\ell)}(x)$ glues with $\sigma$ to a linear character of $I_{\G(\lri_\ell)}(\sigma)$ that extends $\sigma$. By Clifford theory every such extension induces irreducibly to $\G(\lri_\ell)$, proving the first assertion. The number of such extensions is $I_{\G(\lri_\ell)}(\sigma)/\K^m_\ell \simeq \tC_{\G(\lri_m)}(x_m)$, and the second assertion follows. 

\smallskip

Assume now that $\ell=2m+1$. We refer the reader to Figure~\ref{fig:ell.odd} which features all the main players of the construction: linear algebra over $\kk$, Lie theory over $\lri$ and group theory. 

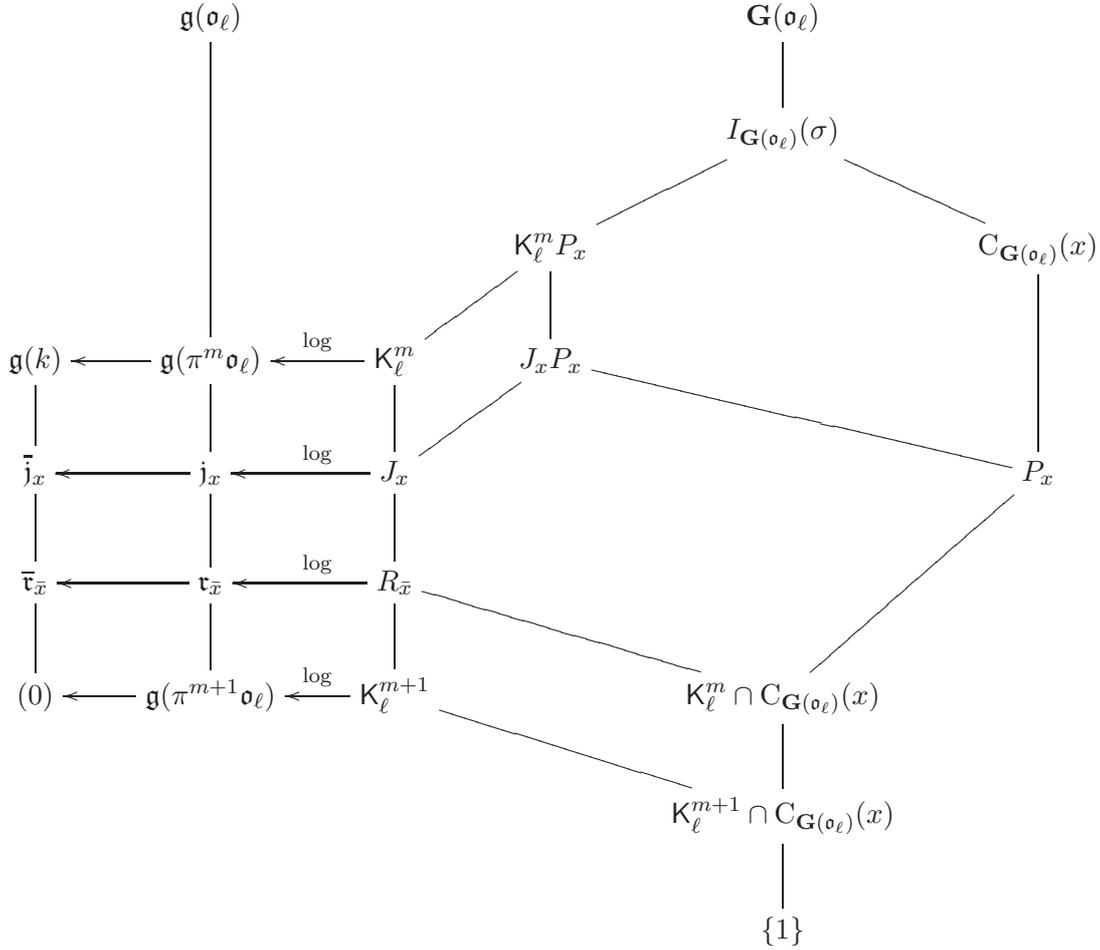
\begin{figure}[htb!]
\centering
 \caption{Odd case $\ell=2m+1$}
 \label{fig:ell.odd}
\begin{displaymath}
\xymatrix{
  & \g(\lri_\ell)\ar@{-}[ddd] &                                &       &     {\G(\lri_\ell)}\ar@{-}[d]       &             \\
   & &                              &    &            {I_{\G(\lri_\ell)}(\sigma)}\ar@{-}[dr] \ar@{-}[dl]     &      \\ 
   & &                              & \K_{\ell}^m P_x \ar@{-}[dl]    \ar@{-}[d]  & & \tC_{\G(\lri_\ell)}(x)\ar@{-}[dd]                     \\   
\g(\kk) \ar@{-}[d] & \ar[l] \g(\pi^m\lri_\ell)   \ar@{-}[d] &  \ar[l]_{\quad \log} {\K^m_\ell}\ar@{-}[d]  &   \ar@{-}[dl]   J_x P_x       \ar@{-}[drr]  &  & \\
\overline{\jj}_x  \ar@{-}[d] & \ar[l] {\mathfrak{j}}_x \ar@{-}[d] &\ar[l]_{\quad \log}  {J_x}\ar@{-}[d]  &                                            &    & P_x \ar@{-}[ddl]  \\
 \overline{\rr}_\xx \ar@{-}[d]  & \ar[l] {\mathfrak{r}}_\xx \ar@{-}[d]& \ar[l]_{\quad \log}   {R_\xx}\ar@{-}[d]  \ar@{-}[drr] &                                            &     \\
(0) &\ar[l] \g(\pi^{m+1}\lri_\ell) & \ar[l]_{\quad \log}    {\K^{m+1}_\ell}\ar@{-}[drr]  &  & {\K^{m}_\ell \cap \tC_{\G(\lri_\ell)}(x)}  \ar@{-}[d]          &\\
  && && {\K^{m+1}_\ell \cap \tC_{\G(\lri_\ell)}(x)}\ar@{-}[d] & \\
   &&  & & {\{1\}}&
}
\end{displaymath}
 \end{figure}

Let $\sigma$ be an irreducible character of $\K^m_\ell$ lying above a regular linear character $\log^*\!\theta$ of~$\K^{m+1}_\ell$, for $\theta \in \g(\pi^{m+1}\lri_\ell)^\vee$, such that $\theta_{\mid \g(\pi^{\ell-1}\lri_\ell)}$ corresponds to a regular element $\bar{x} \in \g(\kk)$. We claim that there exists a lift $x \in \g(\lri_\ell)$ of $\bar{x}$ such that 
\begin{equation}\label{eq:stab.sigma}
I_{\G(\lri_\ell)}(\sigma)=\tC_{\G(\lri_\ell)}(x)\K^m_\ell.
\end{equation}

Indeed, let $\theta'$ be the unique extension of $\theta$ to $\rr_\xx$ that corresponds to $\sigma$; see Section~\ref{subsec:reps.of.K}. Choose $x \in \g(\lri_\ell)$ such that  $\varphi=\varphi_x \in \g(\lri_\ell)^\vee$ satisfies $\varphi_{\mid \rr_\xx}=\theta'$. Since $x$ is regular we may apply Lemma~\ref{lem:form.of.stab} with $r=m$ and $\theta=\varphi_{\mid \g(\pi^{m+1}\lri_\ell)}$, 
and get 
\[ 
I_{\G(\lri_\ell)}(\sigma) \subset I_{\G(\lri_\ell)}(\log^*\theta) =\tC_{\G(\lri_\ell)}(x)\K^m_\ell.
\] 
For the reverse inclusion recall that $\theta'$ determines $\sigma$ and vice versa, therefore
\[
I_{\G(\lri_\ell)}(\sigma) \supset I_{ I_{\G(\lri_\ell)}(\log^*\!\theta)}(\sigma) =I_{ I_{\G(\lri_\ell)}(\log^*\!\theta)}(\log^*\!\theta') \supset \tC_{\G(\lri_\ell)}(x)\K^m_\ell,
\]  
and \eqref{eq:stab.sigma} follows; we remark that the last inclusion, which {\em a posteriori} is equality, follows from the fact that $\tC_{\G(\lri_\ell)}(x)$ stabilises $\varphi_x$ and in particular it stabilises its restriction~$\theta'$, and the fact that $\K^m_\ell$ stabilises $\theta'$ by the Heisenberg lift construction.  

We now show that $\sigma$ extends to its inertia group $I_{\G(\lri_\ell)}(\sigma)$. By Corollary (11.31) in \cite{Isaacs} it is enough to show that $\sigma$ extends to all the subgroups $Q$ such that $Q/\K^m_\ell$ is a Sylow subgroup of $I_{\G(\lri_\ell)}(\sigma)/\K^m_\ell$. For all primes  other than $p$ the representation $\sigma$ extends to the corresponding Sylow subgroup $Q$ since $Q/\K^m_\ell$ is prime to $p$. Let $P_x$ denote the $p$-Sylow subgroup of $\tC_{\G(\lri_\ell)}(x)$. By Lemma \ref{lem:stable.max.iso} there exists a maximal isotropic subspace of $\overline{\rr}_\xx \subset \overline{\jj}_x \subset \g(\kk)$ which is $P_x$-invariant. It follows that its inverse image $\jj_x$ and $\varphi_{\mid \jj_x}$, both of which are $P_x$-invariant, may be used to construct $\sigma$ as in Section~\ref{subsec:reps.of.K}(ii).   
Since $P_x$ is abelian and fixes both $J_x$ and the character $\log^*(\varphi_{\mid \jj_x}):J_x \to \C^\times$, the restriction of $\varphi$ to $J_x \cap P_x=R_\xx \cap \tC_{\G(\lri_\ell)}(x)$,  extends to $P_x$ and any such extension glues to a linear character $\omega$ of $J_x P_x$. The character $\sigma':=\Ind_{J_x P_x}^{\K^m_\ell P_x}(\omega)$ has degree 
$$\sigma'(1)=[\K^m_\ell P_x :J_xP_x ]=[\K^m_\ell:J_x]=\sigma(1)$$ and is therefore an extension of $\sigma$.

It follows that $\sigma$ extends to its inertia group and induces irreducibly to $\G(\lri_\ell)$, proving the first assertion.  To prove the second assertion we note that the possible extensions $\sigma'$ of $\sigma$ are in bijection with $\tC_{\G(\lri_\ell)}(x)/\K^m_\ell \cap \tC_{\G(\lri_\ell)}(x) \simeq \tC_{\G(\lri_m)}(x_m)$.{\hfill $\square$}

\begin{remark} We note that the order of choices made in the proof of Theorem~\ref{thm:bijection} is crucial. We start with $\sigma$, which determines $\theta'$, and only then choose $x$, which determines the extension $\varphi=\varphi_x$. The choice of $x$ determines the decomposition of $I_{G(\lri_\ell)}(\sigma)$ as a product of the centralizer of $x$ and the congruence subgroup and, in particular, determines~$P_x$. Only then $\jj_x$ is chosen as a submodule fixed by $P_x$. 
\end{remark}

\begin{porism}\label{porism} Under the additional assumption that $p=\mathrm{char}(\kk) \nmid n$, Theorem~\ref{thm:bijection} holds for $\SL_n$ and $\SU_n$.

\end{porism}

\begin{proof} If $(p,n)=1$ the determinant map $\det:\K_ \ell^{j} \to (1+\pi^j\Lri_\ell)$ splits for every $1 \le j \le \ell$, since $\K_ \ell^{1}$ is a $p$-group and the restriction of the determinant map to scalar matrices is the invertible map $aI \mapsto a^n$.  It follows that $$\K^j_\ell \simeq (\K')_\ell^j \times \{\text{scalars}\},$$
where $(\K')_\ell^j =\K^j_\ell  \cap \SL_n(\Lri_\ell)$. 
\smallskip

Consequently  $\sigma':=\Res^{\K^m_\ell}_{(\K')_\ell^m}(\sigma)$ is irreducible and all irreducible representations of $(\K')_\ell^m$ that lie above a regular linear character are obtained in this way. Moreover
\[ 
I_{\G'(\lri_\ell)}(\sigma')= I_{\G(\lri_\ell)}(\sigma)  \cap \SL_n(\Lri_\ell),   
\]
hence the extendability of $\sigma'$ to its inertia follows from the extendability of $\sigma$ to its inertia.
\end{proof}


\section{Regular elements in $\g(\kk)$}\label{sec:regulars.in.g.k}

In this section there is no restriction on the characteristic of $\kk$.
Let $\kkalg$ be an algebraic closure of $\kk$. For $d \in \N$, let $\kkd \subset \kkalg$ denote the unique extension of $\kk$ of degree $d$.  Let $\nu \in \Gal(\kkalg/\kk)$ be the Frobenius automorphism $\nu(x)=x^q$. We maintain the notation $\circ$ for the restriction of $\nu$ to $\kkq$ and its extension to an involution of $\M_n(\kkq)$ as a $\kk$-algebra.

\subsection{Description and enumeration of regular elements in $\g(\kk)$}   Let $\sim$ be the multiplicative involution on
$\kkq[t]$ defined by 
\[
h(t) =\sum_i c_i t^i \mapsto \tilde{h}(t)=\sum_i (-1)^{\deg(h)-i}c_i^\circ t^i.
\]
By \cite[Lemma 3.5]{AKOV2} a regular matrix in $\M_n(\kkq)$ is similar to an anti-hermitian matrix if and only if its characteristic polynomial is fixed under this involution. Let 
\[
\begin{split}
\cP_{\gl}(\kkd)&=\{ f \in \kkd[t] \mid \text{$f$ monic and irreducible}\}, ~\text{and}\\
\cP_{\gu}(\kk)&=\{ f \in \kkq[t] \mid \text{$f \in \cP_{\gl}(\kkq)$ and $\tilde{f}=f$, or $f=g \cdot \tilde{g}$ with  $g \ne \tilde{g} \in \cP_{\gl}(\kkq)$}\}.
\end{split}
\]
Let $\cM_\g(\kk)$ be the multiplicative monoid generated by $\cP_\g(\kk)$. It follows that for every $a \in \g(\kk)$, its characteristic polynomial $\Delta_a$ can be written uniquely (up to reordering) in the form $\Delta_a= \prod_{i} f_i^{e_i}$, for some finite list of elements $f_i \in \cP_\g(\kk)$ with $\sum_{i} \deg(f_i) e_i=n$, namely, that $\Delta_a \in \cM_\g(\kk)$.

\begin{definition}\label{def:n.typical.matrix}
We call a matrix $\ty \in M_{n}(\mathbb{Z}_{\ge 0})$ $n$-typical if
$\sum_{d,e} d e \ty_{d,e}=n$. We denote the set of $n$-typical matrices by $\cA_n$.
\end{definition}

For $h \in \cM_\g(\kk)$ of degree $n$ let $\ty_{d,e}(h)$ be the number of distinct polynomials $f \in \cP_\g(\kk)$ occurring in $h$ such that $d=\deg(f)$ and $e$ is the multiplicity
of $f$ in $h$. We get an $n$-typical matrix which we call the {\em type} of
$h$. For $\ty \in \cA_n$ let $\cM_\g(\kk)^\tau$ be the subset of polynomials in $\cM_\g(\kk)$ of type $\tau$. Let 
\[
w_d(x)=\frac{1}{d}\sum_{m|d}\mu(d/m)x^{m} \in \Q[x].
\]

\begin{proposition}\label{prop:number.of.typical.pol} Let $n \in \N$ and assume that $|\kk|=q > n$. For any $\ty \in \cA_n$ we have 
\[
\left| \cM_\g(\kk)^\ty \right| = \prod_{d } \left(\sum_{e} \tau_{d,e} \atop \tau_{d,1}, \tau_{d,2}, \ldots\right) \left(w_d(q) \atop \sum_e \tau_{d,e}\right).
\]
In particular $\left| \cM_\gu(\kk)^\ty \right|=\left| \cM_\gl(\kk)^\ty \right|$.
\end{proposition}

To prove Proposition~\ref{prop:number.of.typical.pol} we require the following Lemma which is the anti-hermitian analogue of \cite[Lemma~2]{Ennola-unitary-conj}.

\begin{lemma}\label{lem:analog.of.ennola} Let $f(t) \in \kkq[t]$ be irreducible with $\tilde{f}(t)=f(t)$. Then the degree of $f(t)$ is odd and any root $\xi$ of $f(t)$ is in the kernel of the trace map $\tr_{\kkqd \mid \kkd} : \kkqd \to \kkd$.
\end{lemma}

\begin{proof}

Assume first that $q$ is odd.  Consider the $\circ$-map on $R=\M_d(\kkq)$. For every matrix $b \in R$ we get an involution of $\kk$-algebras $\circ: \tC_R(b) \to \tC_R(b^\circ)$. By~\cite[Lemma 3.5]{AKOV2} there exists $\xi \in R$ whose characteristic polynomial is $f(t)$ and $\xi+\xi^\circ=0$. This, combined with the assumption that $f$ is irreducible, means that $\tC_R(\xi)=\tC_R(\xi^\circ)=\kkq[\xi] \simeq \kkqd$. We thus have the following commutative diagram
\[
\begin{matrix} R & \overset{\circ}{\longrightarrow} & R \\
\cup & & \cup \\
\tC_R(\xi) & \longrightarrow &\tC_R(\xi)\end{matrix}
\]
and the restriction of $\circ$ to $\tC_R(\xi)$ is a field automorphism of order two fixing $\kk$ and hence of the form $\eta \mapsto \eta^{q^d}$ for all $\eta \in \tC_R(\xi)$. It follows that $\tr_{\kkqd \mid \kkd}(\xi)=\xi+\xi^{q^d}=\xi+\xi^\circ=0$. 

It remains to show that $d$ is odd. We use the unitary version \cite[Lemma 2]{Ennola-unitary-conj} and the Cayley map. Recall the Cayley map $\mathrm{Cay}: \gu_d \to \GU_d$ defined by $a \mapsto (1+a)(1-a)^{-1}$.  This is a rational map defined only on elements $a$ that do not have $-1$ as an eigenvalue. Nevertheless, on elements of interest to us, namely generators of $\kkqd$ as an $\kkq$-algebra, it is always defined. Observing that $\mathrm{Cay}$ commutes with the Galois action we deduce that $\xi$ generates $\kkqd$ over $\kkq$ if and only if $\mathrm{Cay}(\xi)$ does. The fact that $d$ is odd now follows from the unitary case.  


Assume now that $q$ is even. Recall that on $\kkq$ the Galois action is $x^\circ=x^q$. The condition $f(t) = \tilde{f}(t)$ implies that both $\xi$ and $\xi^q$ are roots of $f(t)$. All the roots of the irreducible polynomial $f(t)$ are conjugate under the group generated by $x \rightarrow x^{q^2}$ and therefore 
$\xi^{q^{2i}} = \xi^q$ for some $i$ satisfying $1 \leq i < d$, where $d = \deg(f)$. This, along with the fact that the exponent of $\xi$ is coprime to $q$,  implies that $\xi^{q^{2i-1 }-1} = 1$. As $\xi$ is a root of an irreducible polynomial of degree $d$ we also have $\xi^{q^{2d}-1} = 1$. Let $e$ be the exponent of $\xi$. We have 
\begin{equation}\label{eqn:cong}
	q^{2d} \equiv 1 \pmod{e} \quad \mathrm{and} \quad q^{2i-1} \equiv 1  \pmod{e}.
\end{equation}
Let $r = \gcd(2i-1, 2d )$. Since $r$ is odd it divides $d$ and therefore we get $q^d \equiv 1 \pmod{e}$. This implies that $\xi^{q^d-1}=1$ and hence $\xi \in \kkd$. It follows that $d$ is odd as otherwise $\xi$ lies in a proper subfield of $\kkqd$ that contains $\kkq$, in contradiction to $\xi$ being a root of degree~$d$ irreducible polynomial over $\kkq$. Finally, we have  $\mathrm{tr}_{\kkqd \mid \kkd} (\xi)=\xi^{q^d}+\xi=2\xi=0$. This completes the proof for $q$ even. \qedhere

\end{proof}

\begin{proof}[Proof of Proposition~\ref{prop:number.of.typical.pol}] First, assume that $\ty_{n,1}=1$ and all other entries are zero, that is, $\ty$ corresponds to an irreducible polynomial of degree $n$. Then, it is well known that $w_n(q)$ is the number of such polynomials over $\kk$, that is $\left|\cM_\gl(\kk)^\ty\right|=w_n(q)$.  For the unitary counterpart we consider the odd and even cases separately. If $n$ is odd then we need to count the number of monic irreducible polynomials of degree $n$ over $\kkq$ which satisfy $\tilde{f}(t)=f(t)$. By Lemma~\ref{lem:analog.of.ennola} we need to count the number of elements in $\kkqn$ which are in the kernel of the trace map $\tr_{\kkqn \mid \kkn}$ and divide by $n$. Thus, as in the derivation of $w_n(q)$, an inclusion-exclusion argument gives that $\left|\cM_\gl(\kk)^\ty\right|=w_n(q)$. If $n$ is even we need to count the number of polynomials $f(t)=g(t)\tilde{g}(t)$ with $g(t)$ and $\tilde{g}(t)$ distinct irreducible polynomials of degree $n/2$ over $\kkq$. The number of such $g$'s is $w_{\frac{n}{2}}(q^2)=2w_n(q)$. Taking into consideration the symmetry between $g$ and $\tilde{g}$ we divide by $2$ and the result follows. The formula for general $\ty$ now follows by multiplying over the various possibilities for the irreducible constituents and making sure that there are no repeated polynomials except for those dictated by $\ty$.  \qedhere

\end{proof}

\subsection{Centralisers in $\G(\kk)$ of regular elements of $\g(\kk)$}

The following proposition describes the structure of centralisers of regular elements. The type of a regular element in $\g(\kk)$ is the type of its characteristic polynomial. We say that a type $\ty \in \cA_n$ is $(d,e)$-primary if $\ty_{d,e}=1$ for $d,e$ with $de=n$ and zero otherwise. That is, the characteristic polynomial of $a$ is of the form $\Delta_a(t)=f(t)^e$ with $f(t) \in \cP_\g(\kk)$ of degree $d$. Let $\HH_{\epsilon,d}$ be the following algebraic group 
\[
\HH_{\epsilon,d}=\left\{
\begin{array}{cl}
 \GL_1, & \text{if $\epsilon=1$ (independent of $d$), or $\epsilon=-1$ and $d$ is even},\\
\GU_{1, \kkqd\mid\kkd}, & \text{if $\epsilon=-1$ and $d$ is odd},
\end{array} \right.
\]
where $\GU_{1, \kkqd\mid\kkd}$ denotes the one-dimensional unitary group associated to the quadratic extension $\kkqd / \kkd$. 
\begin{proposition}\label{prop:form.of.centralizers} Let $\tC_{\G(\kk)}(a)$ denote the centraliser in $\G(\kk)$  of a regular element $a \in \g(\kk)$ of type $\ty$. Then 
\[
\tC_{\G(\kk)}(a) \simeq \prod_{d,e} \mathbf{H}_{\ee_\G,d}\left(\kkd[t]/(t^e)\right)^{\ty_{d,e}}.
\]
In particular, the cardinality of the centraliser is $u_{\ee_\G}^\ty (q)$.

\end{proposition}

\begin{proof}  Let $a \in \g(\kk)$ be regular of type $\ty$. Assume first that $\ty$ is $(d,e)$-primary and write $\Delta_a(t)=f(t)^e$ with $f(t) \in \cP_\g(\kk)$ of degree $d$. If $\g=\gl_n$ then $f$ is irreducible over $\kk$ and $\tC_{\M_n(\kk)}(a) = \kk[a] \simeq \kkd[u]/(u^e)$ and the result  follows by taking units. 

If $\g=\gu_n$ then $f$ is of degree $d$ over $\kkq$ and satisfies $f = \tilde{f}$. Assume first that $d$ is odd which means that $f$ is irreducible.  Write $a=S+N$, the additive Jordan decomposition of $a$, with $S$ semisimple and $N$ nilpotent such that $[S,N]=0$. Then we have that $S,N \in \g(\kk)$, and $\tC_{\M_n(\kkq)}(a)=\kkq[a]=\kkq[S][N] \simeq \kkqd[u]/(u^e)$.  A short computation shows that intersecting the latter with $\G(\kk)$ gives $\tC_{\G(\kk)}(a) \simeq \GU_{1,\kkqd/\kkd}(\kk[u]/(u^e))$. Now assume that $d$ is even and so $f$ is a product of two distinct irreducible polynomials $g$ and $\tilde{g}$ of degree $d/2$.  Let $\alpha \in \M_{d/2}(\kkq)$ be a root of $g(t)$ so that $\kkd \simeq \kkq[\alpha] \hookrightarrow \M_{d/2}(\kkq)$. Then $-\alpha^\circ$ is a root of $\tilde{g}(t)$ and $a$ is similar in $\M_2\!\left(\M_{e}(\M_{d/2}(\kkq))\right)=\M_n(\kkq)$
 to the block diagonal matrix $\mathrm{diag}(J_e(\alpha),-J_e(\alpha)^\circ)$, where
\begin{equation*}
  J_e(\alpha)=\left(
    \begin{array}{ccccc}
      \alpha & \mathrm{I}      &    &     0\\
      0      & \alpha & \ddots &   \\
      \vdots & \ddots & \ddots & \mathrm{I}\\
      0 & \cdots & 0 &  \alpha\\
        \end{array}
  \right) 
  . 
\end{equation*}
A short computation shows that $\tC_{\G(\kk)}(a) \simeq \GL_1(\kkd[u]/(u^e))$.

Finally, for general $\ty$, if $\prod f_i(t)^{r_i}$ is a decomposition with distinct $f_i \in \cP_\g(\kk)$, then the centraliser of $a$ is the product of the centralisers of the primary parts.  \qedhere

\end{proof}

\subsection{Centralisers in $\G(\kk) \cap \SL_n(\kkq)$ of regular elements in $\g(\kk)$} Let $\G' : = \G \cap \SL_n$. Recall that for $\ty \in \cA_n$, with associated characteristic polynomial $\prod_{i=1}^m f_i(t)^{r_i}$ such that $f_i \in \cP_\g(\kk)$ are distinct, we defined $\iota(\ty,q-\ee)=\gcd(r_1,...,r_m,q-\ee)$.

\begin{proposition}\label{prop:centralizer.G'} Let $a \in \g(\kk)$ be a regular element of type $\ty \in \cA_n$.  Then 
\begin{equation*}
|\tC_{\G(\kk)}(a):\tC_{\G'(\kk)}(a)|= \frac{q-\ee_\G}{\iota(\ty, q-\ee_\G)}.
\end{equation*}
\end{proposition}

\begin{proof}  Let $\G_1$ be either $\GL_1$ or $\GU_1$ according to $\G$ being $\GL_n$ or $\GU_n$, respectively. We have an exact sequence 
\begin{equation}\label{exact.seq.of.centralizers}
1 \to \tC_{\G'(\kk)}(a) \hookrightarrow \tC_{\G(\kk)}(a) \overset{\det}{\longrightarrow} \G_1(\kk).
\end{equation}
As $|\G_1(\kk)|=q-\ee_\G$, we may equally well show that the size of the co-kerel of the determinant map is $\iota(\ty, q-\ee_G)$. Suppose first that $\ty$ is $(n,1)$-primary. In such case the statement boils down to exactness of the rightmost arrow in \eqref{exact.seq.of.centralizers}. If $\g=\gl_n$ this means that~$\ty$ corresponds to an irreducible polynomial of degree $n$ over $\kk$. In such case $\tC_{\G(\kk)}(a) \simeq \kkn^\times$ and the determinant map identifies with the norm map $N_{\kkn \mid \kk}: \kkn^\times \to \kk^\times$ which is surjective.  If $\g=\gu_n$ we consider the two cases according to the parity of $n$. If $n$ is odd then $\tC_{\GU_n(\kk)}(a) \simeq \ker N_{\kkqn \mid \kkn}$ and the determinant map identifies with the norm map $N_{\kkqn \mid \kkq}$. Its surjectivity on $\GU_1(\kk) = \ker N_{\kkq \mid \kk}$ follows from the surjectivity of the norm maps and Hilbert 90. For $n$ even the characteristic polynomial of $a$ is a product $g(t)\tilde{g}(t)$, with $g(t) \ne \tilde{g}(t)$ irreducible polynomials of degree $n/2$ over $\kkq$. We have $\tC_{\G(\kk)}(a) \simeq \kkn^\times$ and the determinant map coincides with $\alpha \mapsto \alpha \cdot \alpha^{-q}$ which is, again, a surjection on $\GU_1(\kk)$. 

For general $\ty$ we get the following diagram 
\[
\begin{matrix}
\prod_{d,e} \HH_{\ee_\G,d}\left(\kkd[t]/(t^e)\right)^{\times \ty_{d,e}} & \simeq & \tC_{\G(\kk)}(a) & \subset & \G(\kk) \\
\eta \twoheaddownarrow  & &\twoheaddownarrow & & \twoheaddownarrow \det \\
\prod_{d,e} \left(\G_1(\kk)^{e}\right)^{\times \ty_{d,e}} & \overset{\phi}{\twoheadrightarrow} & \G_1(\kk)^{\gcd \{e \,\mid\, \exists d ~\ty_{d,e} \ne 0\}} & \subset & \G_1(\kk), 
\end{matrix}
\]
where the map $\eta$ is composition of reduction modulo the radical and component-wise application of appropriate norm maps; the map $\phi$ is the group multiplication. For every primary factor $f(t)^e$ of the characteristic polynomial $\Delta_a(t)$ of $a$, the image of the restriction of the determinant map to the corresponding block in $\tC_{\G(\kk)}(a)$ is the subgroup of $e$-powers $\G_1(\kk)^e$ in the cyclic group $\G_1(\kk)$. Going over all the possible exponents $e$ that occur in $\tau$ we get that the image of the determinant map is $\G_1(\kk)^{\gcd \{e \,\mid\, \exists d ~\ty_{d,e} \ne 0\}}$, as required.   \qedhere

\end{proof}

\section{Dimensions, enumeration and Ennola duality}\label{sec:sim.enum.ennola}

We are now in a position to prove the remaining theorems and the corollary stated in the introduction.

\begin{proof}[Proof of Theorem~\ref{thm:dimensions}] We use the proof of Theorem~\ref{thm:bijection}. Regular characters of $\G(\lri)$ of level $\ell$ are characters of $\G(\lri_{\ell+1})$. Let $\chi$ be such character. 

Assume first that $\ell+1=2m$. Then $\chi$ is induced from a linear character of the stabiliser $I_{\G(\lri_{2m})}(\sigma)=\tC_{\G(\lri_{2m})}(x)\K^m_{2m}$ that extends some linear character $\sigma$ of $\K^m_{2m}$ and $x \in \g(\lri_{2m})$, see Figure~\ref{fig:ell.even}. The degree is therefore 
\begin{equation}\label{eq:dim.rho.even}
\begin{split}
\chi(1) &= [\G(\lri_{2m}): I_{\G(\lri_{2m})}(\sigma)]=\frac{|\G(\lri_{m})|}{|\tC_{\G(\lri_m)}({x}_m)|} \\ &=\frac{q^{n^2(m-1)}|\G(\kk)|}{q^{n(m-1)}|\tC_{\G(\kk)}({x}_1)|}=q^{\left( n \atop 2 \right)(\ell-1)} \frac{v_{\ee_\G}(q)}{u_{\ee_\G}^\ty(q)},
\end{split}
\end{equation}
where $x_j$ denotes the reduction of $x$ modulo $\pi^j$ and $\ty$ is the type of ${x}_1$. 

\smallskip

If $\ell+1=2m+1$, the construction differs by a Heisenberg lift, that is, $\chi$ is induced from a character of $I_{\G(\lri_{\ell+1})}(\sigma)$ of degree $q^{\left( n \atop 2 \right)}$, yielding
\begin{equation}\label{eq:dim.rho.odd}
\begin{split}
\chi(1) &= q^{\left( n \atop 2 \right)}[\G(\lri_{\ell+1}): I_{\G(\lri_{\ell+1})}(\sigma)] =q^{\left( n \atop 2 \right)}\frac{|\G(\lri_{m})|}{|\tC_{\G(\lri_m)}({x}_m)|} \\ &=q^{\left( n \atop 2 \right) (2m-2 +1)}\frac{|\G(\kk)|}{|\tC_{\G(\kk)}({x}_1)|}=q^{\left( n \atop 2 \right)(\ell-1)} \frac{v_{\ee_\G}(q)}{u_{\ee_\G}^\ty(q)}.
\end{split}
\end{equation} \qedhere

\end{proof}

\begin{proof}[Proof of Theorem~\ref{thm:numbers}]

Let $\ty \in \cA_n$ be a type. By Proposition~\ref{prop:number.of.typical.pol} the number of distinct polynomials of type $\ty$  is 
\[
\left| \cM_\g(\kk)^\ty \right|=\prod_{d } \left(\sum_{e} \tau_{d,e} \atop \tau_{d,1}, \tau_{d,2}, \ldots\right) \left(w_d(q) \atop \sum_e \tau_{d,e}\right). 
\]

It follows that this is the number of $\G(\kk)$-orbits of regular elements in $\g(\kk)$ of type $\ty$, and hence also the number of $\G(\kk)$-orbits of regular characters in $\g(\kk)^\vee$ of type $\ty$, see Section~\ref{subsec:duality.for.lie.rings}. 

Assume first that $\ell+1=2m$. Characters in each of these orbits, viewed as characters of $\K^{\ell}/\K^{\ell+1}$, extend to $\K^m/\K^{\ell+1}$. By regularity, the number of inequivalent extensions of each character of $\K^{\ell}/\K^{\ell+1}$ is $q^{n(\ell-m)}$.  By Theorem~\ref{thm:bijection} they all extend further to the stabiliser~in 
\[
\left|I_{\G(\lri_{\ell+1})}(\sigma)/\K^m_{\ell+1}\right|=q^{n(m-1)}u_{\ee_\G}^\ty(q)
\]
different ways, and they all induce to inequivalent irreducible character. It follows that the number of regular representations of $\G(\lri_{\ell+1})$ of type $\ty$ is  
\begin{equation}\label{eq:counting.irreps.of.type.ty}
\begin{split}
\mathrm{Irr}(\G(\lri_{\ell+1}); \ty) & =  q^{n(m-1)} \cdot \left| I_{\G(\lri_{\ell+1})}(\sigma) /\K^m_{\ell+1} \right| \cdot \left| \cM_\g(\kk)^\ty \right|  \\ &=q^{n(\ell-1)}u_{\ee_\G}(q) \prod_{d } \left(\sum_{e} \tau_{d,e} \atop \tau_{d,1}, \tau_{d,2}, \ldots\right) \left(w_d(q) \atop \sum_e \tau_{d,e}\right).
\end{split}
\end{equation}

Similar considerations give the formula for $\ell+1=2m+1$ odd. The only twist is that one needs to take into account the $q^n$ inequivalent extensions of linear characters of $K^{m+1}_{\ell+1}$ to $R_\xx$, see Figure~\ref{fig:ell.odd}.    \qedhere
 
\end{proof}

\begin{proof}[Proof of Corollary~\ref{cor:zeta.G}] Combining Theorems~\ref{thm:dimensions}~and~\ref{thm:numbers} gives
\[
\begin{split}
\zeta^{\mathrm{reg}}_{\G(\lri_{\ell+1})}(s) &= \sum_{\ty \in \cA_n} \sum_{\chi \in \mathrm{Irr}(\G(\lri_{\ell+1}); \ty)} \chi(1)^{-s} \\
&=q^{(\ell-1)\left(n-\left(n \atop 2\right) s\right)}\sum_{\ty \in \cA_n} u_{\ee_\G}^\ty(q) \prod_{d } \left(\sum_{e} \tau_{d,e} \atop \tau_{d,1}, \tau_{d,2}, \ldots\right) \left(w_d(q) \atop \sum_e \tau_{d,e}\right) \frac{v_{\ee_\G}(q)^{-s}}{u_{\ee_\G}^\ty(q)^{-s}}.
\end{split}
\]

The corollary follows by $\zeta^{\mathrm{reg}}_{\G(\lri)}(s)=\sum_{\ell=1}^\infty \zeta^{\mathrm{reg}}_{\G(\lri_{\ell+1})}(s)$. \qedhere

\end{proof}

\begin{proof}[Proof of Theorem~\ref{thm:zeta.H}] By Porism~\ref{porism} the construction theorem (Theorem~\ref{thm:bijection}) continues to hold upon replacing $\G$ by $\G'$. To get the regular representation zeta function of $\G'(\lri)$ we follow the proofs of Theorems~\ref{thm:dimensions}~and~\ref{thm:numbers} and replace the occurrences of $\G$ by $\G'$ and of $\g$ by $\g':=\g \cap \mathfrak{sl}_n$. We now address the main points that need further justification.

\smallskip

First, the occurrences of $|\G(\kk)|/|\tC_{\G(\kk)}({x}_1)|$ in equations \eqref{eq:dim.rho.even} and  \eqref{eq:dim.rho.odd} are replaced by $|\G'(\kk)|/|\tC_{\G'(\kk)}({x}_1)|$. To justify the latter we need to show that the restriction of the reduction map  $\G(\lri_m) \to \G(\kk)$ to 
\begin{equation}\label{eqn:surjectivitiy.of.centralisers}
\tC_{\G'(\lri_m)}(x) \to \tC_{\G'(\kk)}(x_1)
\end{equation}
is surjective. Since the map $\tC_{\G(\lri_m)}(x) \to \tC_{\G(\kk)}(x_1)$ is surjective, for every $g_1 \in \tC_{\G'(\kk)}(x_1)$ there exists a preimage $g \in \tC_{\G(\lri_m)}(x)$. This means that $\det(g) \in 1+\pi\Lri_m$. Since $p$ is prime to $n$ we can find a scalar matrix $aI \in \K^1_m$ with $\det(g)=\det (aI)$ so that $g'=ga^{-1}I \in  \tC_{\G'(\lri_m)}(x)$ and $g'  \equiv_\pi g_1$, hence the subjectivity of the reduction map~\eqref{eqn:surjectivitiy.of.centralisers}.

Therefore, consulting Proposition~\ref{prop:centralizer.G'}, we deduce that the degree of a regular character $\chi$ of $\G'(\lri_{\ell+1})$ lying above characters of type $\ty$ is
\[
\chi(1)=q^{\left( n \atop 2 \right)(\ell-1)} \frac{v_{\ee_\G}(q)}{u_{\ee_\G}^\ty(q) \iota(\ty,q-\ee_\G)}.
\]

\smallskip

Second, we address the enumeration of such characters. The trace form on $\M_n(\kkq)$ restricts to a non-degenerate paring between $\g'(\kk)$ and $\g'(\kk)^\vee$ as long as we require that the characteristic is prime to $n$. We then get an equivariant isomorphism between the two. The well-defined surjective map $\g(\kk) \to \g'(\kk)$ given by $a \mapsto a -\frac{1}{n}\tr(a)I_n$,  implies that  the types that occur in $\g'(\kk)$ are the same as those occurring in $\g(\kk)$, however, the number of classes of every type $\ty$ is reduced by a factor of $q$, namely
\begin{equation}\label{eq:types.for.g'}
\left| \cM_{\g'}(\kk)^\ty \right|=q^{-1} \left| \cM_{\g}(\kk)^\ty \right|.
\end{equation} 
We also note that each $\G(\kk)$-class in $\g'(\kk)$ splits into $\iota(\ty,q-\ee_\G)$ distinct $\G'(\kk)$-classes of equal cardinality since for every $a \in \g(\kk)$ we have 
\[
\frac{\left|\G(\kk)a\right|}{\left|\G'(\kk)a\right|}=\frac{[\G(\kk):\tC_{\G(\kk)}(a)]}{[\G'(\kk):\tC_{\G'(\kk)}(a)]}=\iota(\ty,q-\ee_\G).
\]

The number of such regular characters is given by the following $\G'$-analogue of \eqref{eq:counting.irreps.of.type.ty} 
\begin{equation*}\label{eq:counting.irreps.of.type.ty.G'}
\begin{split}
\mathrm{Irr}(\G'(\lri_{\ell+1}); \ty) & =  q^{(n-1)(m-1)} \cdot \left| I_{\G'(\lri_{\ell+1})}(\sigma) /(\K')^m_{\ell+1} \right| \cdot \left| \cM_{\g'}(\kk)^\ty \right|  \\ &=q^{(n-1)(\ell-1)}u_{\ee_\G}(q) \frac{\iota(\ty,q-\ee_\G)}{q-\ee_\G} \cdot q^{-1}\prod_{d } \left(\sum_{e} \tau_{d,e} \atop \tau_{d,1}, \tau_{d,2}, \ldots\right) \left(w_d(q) \atop \sum_e \tau_{d,e}\right).
\end{split}
\end{equation*}

Consequently, the regular zeta function of $\G'(\lri_{\ell+1})$ is 
\[
\begin{split}
\zeta^{\mathrm{reg}}_{\G'(\lri_{\ell+1})}(s) &= \sum_{\ty \in \cA_n} \iota(\ty,q-\ee_\G) \sum_{\chi \in \mathrm{Irr}(\G'(\lri_{\ell+1}); \ty)} \chi(1)^{-s} \\
&=q^{(\ell-1)\left(n-1-\left(n \atop 2\right) s\right)}\sum_{\ty \in \cA_n} \frac{\iota(\ty,q-\ee_\G)^2}{(q-\ee_\G)} u_{\ee_\G}^\ty(q) \\ & \qquad \qquad \qquad  \times q^{-1}\prod_{d } \left(\sum_{e} \tau_{d,e} \atop \tau_{d,1}, \tau_{d,2}, \ldots\right) \left(w_d(q) \atop \sum_e \tau_{d,e}\right) \left(   \frac{v_{\ee_\G}(q)}{u_{\ee_\G}^\ty(q)\iota(\ty,q-\ee_\G)}   \right)^{-s},
\end{split}
\]

and finally, we have $\zeta^{\mathrm{reg}}_{\G'(\lri)}(s)=\sum_{\ell=1}^\infty \zeta^{\mathrm{reg}}_{\G'(\lri_{\ell+1})}(s)$.   \qedhere

\end{proof}


\def\cprime{$'$}

\end{document}